\documentclass[11pt]{amsart}
\usepackage{graphicx}
\usepackage[english]{babel}
\usepackage{amsthm}
\usepackage{float} 
\usepackage{amsmath}
\usepackage{amssymb}
\usepackage{amsfonts}
\usepackage{color}
\usepackage{algorithm}
\usepackage{algorithmic}
\usepackage{enumerate}
\usepackage{cite}
\usepackage{hyperref}
\usepackage{stmaryrd}
\usepackage{caption}
\usepackage{subcaption}
\usepackage{mathrsfs}
\usepackage{enumitem}
\usepackage{mathrsfs}  

\usepackage{geometry}

\numberwithin{equation}{section}

\allowdisplaybreaks[3]

\newtheorem{thm}{Theorem}[section]
\newtheorem{lemma}[thm]{Lemma}

\newtheorem{prop}[thm]{Proposition}
\theoremstyle{definition}
\newtheorem{definition}[thm]{Definition}
\theoremstyle{remark}

\usepackage{array}

\usepackage{blindtext}

\newcommand{\comment}[1]{}

\usepackage{bm}

\newcommand{\p}{\partial}

\newcommand{\calT}{\mathcal{T}}

\newcommand{\MJN}[1]{\textcolor{black}{#1}}

\title{A note on the shape regularity of Worsey-Farin splits}
\author[S. Gong]{Sining Gong}
\address{Division of Applied Mathematics, Brown University, Providence, RI 02912 }
\email{sining\_gong@brown.edu}

\author[J. Guzm\'an]{Johnny Guzm\'an}
\address{Division of Applied Mathematics, Brown University, Providence, RI 02912 }
\email{johnny\_guzman@brown.edu}           
\thanks{
The second author was supported in part by NSF grant DMS--1913083.  The third
author was supported in part by NSF grant DMS--2011733.}

\author[M. Neilan]{Michael Neilan}   
\address{Department of Mathematics, University of Pittsburgh, Pittsburgh, PA 15260}
\email{neilan@pitt.edu}

\begin{document}

\maketitle

\begin{abstract}
We prove \MJN{three-dimensional} Worsey-Farin refinements
inherit their parent triangulations' shape regularity.
\end{abstract}

\thispagestyle{empty}

\section{Introduction}
Three-dimensional Worsey-Farin splits
were first introduced in \cite{worsey1987ann} 
to construct low-order $C^1$ splines
on simplicial triangulations,
and  they have been extensively studied since then; see for example \cite{lai2007spline}. 
 Recently it has been shown that  smooth piecewise polynomial spaces
 on Worsey-Farin splits (and related ones) fit into discrete de Rham \MJN{complexes}.  
 \cite{guzman2020exact,guzman2022exact, hu2022partially, christiansen2018generalized, fu2020exact}. 
 These results are further applied to analyze convergence, stability and accuracy of numerical methods for models of incompressible fluids on theses 
 refinements \cite{CLM2021,max3D2022,KNS22}. Therefore, it is necessary to discuss the properties of these refinements, 
 especially {in the context of} approximation and stability properties of the corresponding discrete spaces. One critical geometric property for approximation theory is the shape regularity of 
 the underlying mesh.    
 
The shape regularity of Worsey-Farin splits
 are required to ensure optimal-order
and uniform interpolation estimates  
in \cite[Theorem 18.15]{lai2007spline}, \cite[Theorem 6.3]{AlfeldSchumaker05}, 
 \cite[Theorem 6.2]{Sorokina09}, and  \cite[Theorem 8.14]{Matt12}.
Stability estimates of a finite element 
method in \cite{Fabienetal22}
defined on Worsey-Farin splits also require
regularity of the refined triangulation.
The references
 \cite[Page 515]{lai2007spline}, \cite[Remark 14]{AlfeldSchumaker05},
 and \cite[Page 54]{AlfeldSorokina09} explicitly conjecture
 that  Worsey-Farin splits
 of a family of shape regular meshes remain shape regular.
However, to the best of our knowledge, 
a proof of this result has not appeared in the literature. In this note we fill in this gap.

  In \cite[Lemma 4.20]{lai2007spline} and \cite[Lemma 2.6]{KNS22},
  the relationship between the shape regularity constant of Powell-Sabin splits and the parent triangulations is shown. Namely, this result is proved by establishing bounds of the angles of each macro triangle. 
  Hence, it is natural to focus on the dihedral angles in the three-dimensional Worsey-Farin case.   We first prove the dihedral angles are bounded by quantities that only  
depend on the shape regularity of the original mesh (see Lemma \ref{lem:cos} below). Using this result we prove the crucial result that the split points of each face $F$ 
  in the triangulation is uniformly bounded away from $\partial F$; see Lemma \ref{lem:splitF}. From this result, the shape regularity of Worsey-Farin refinements is then shown.

This paper is organized as follows. In Section \ref{sec:Prelim}, we recall the Worsey-Farin refinement of a
three-dimensional simplicial mesh and present some notations to better illustrate our main analysis. In Section \ref{sec:analysis}, we show the shape regularity of Worsey-Farin splits is 
solely determined by the shape regularity of the parent mesh. 

\section{Preliminaries} \label{sec:Prelim}
\subsection{Geometric notations and properties}
We first present some basic definitions regarding the geometric properties of a tetrahedron, see \cite[Definition 16.1-16.2]{lai2007spline} for more details.

Given a tetrahedron $T$, we
denote by $\Delta_m(T)$ the set of $m$-dimensional simplices 
of $T$.  For example, $\Delta_2(T)$ is the set of four faces of 
$T$, and $\Delta_1(T)$ is the set of six edges of $T$.
 Let $\rho_T$ be the diameter of the inscribed sphere $S_T$ of $T$,
which  is the largest sphere contained in $T$. We call the center of $S_T$ the incenter of $T$, denoted by $z_T$, and call the radius of $S_T$ the inradius of $T$, equal to $\rho_T/2$. 
The sphere $S_T$ intersects each face $F$ of $T$ at a unique point, $z_{T,F}$.  
We note that $z_{T,F}$ is the orthogonal projection of the point $z_T$ to the plane that contains $F$ (i.e., the vector $z_T-z_{T,F}$ is normal to $F$).  
Finally, we let $h_T = \text{diam} (T)$.

The following two propositions are well-known results of tetrahedra.  To be self-contained we provide their proofs.
\begin{prop} \label{prop:rhoK}
For a tetrahedron $T$, there holds
\[ 
\rho_T = 6|T|/(\sum\limits_{F\in \Delta_2(T)} |F|).
\]
\end{prop}
\begin{proof}
Consider the refinement of $T$ obtained
by connecting the incenter of $T$ to its vertices.
The resulting four subtetrahedra fill the volume of 
$T$, and thus,
\[ 
|T| = \sum\limits_{F\in \Delta_2(T)} \frac{1}{3} |F| \frac{\rho_T}{2},
\]
which gives the result.
\end{proof}

\begin{prop} \label{prop:height}
Given a tetrahedron $T$, let $x$ be any vertex of $T$ and $F_x$ be the face of $T$ which is opposite to $x$. Let $P_x$ be the plane containing $F_x$,  then for any point $a \in \textnormal{int}(T)$, we have 
\begin{equation}\label{312}
\textnormal{dist}(x, P_x)> \textnormal{dist}(a,P_x). 
\end{equation}
In particular,  
\begin{equation}\label{313}
 \textnormal{dist}(x, P_x)> \rho_T. 
 \end{equation}
\end{prop}
\begin{proof}
Since the point $a \in \textnormal{int}(T)$ and $F_x$ is a face of $T$, $a$ and $F_x$ form a tetrahedron $T' \subset T$. Therefore,
\[ 
\frac{1}{3}|F_x|\textnormal{dist}(x, P_x) = \textcolor{black}{|T|}>\textcolor{black}{|T'|}= \frac{1}{3}|F_x|\textnormal{dist}(a, P_x),
\]
which immediately gives  \eqref{312}.
Let $\ell$ be the line containing $z_T$ and $z_{T,F_x}$. Let $\ell$ intersect $S_T$ at $a \neq z_{T,F_x}$. Then $a \in \text{int}(T)$ and $\textnormal{dist}(a,P_x)=|[a,z_{T,F_x}]| =\rho_T$. Hence, \eqref{313} follows from \eqref{312}.
\end{proof}

\begin{figure}
      \centering
      \begin{subfigure}[b]{0.45\textwidth}
          \centering
          \includegraphics[width=\textwidth]{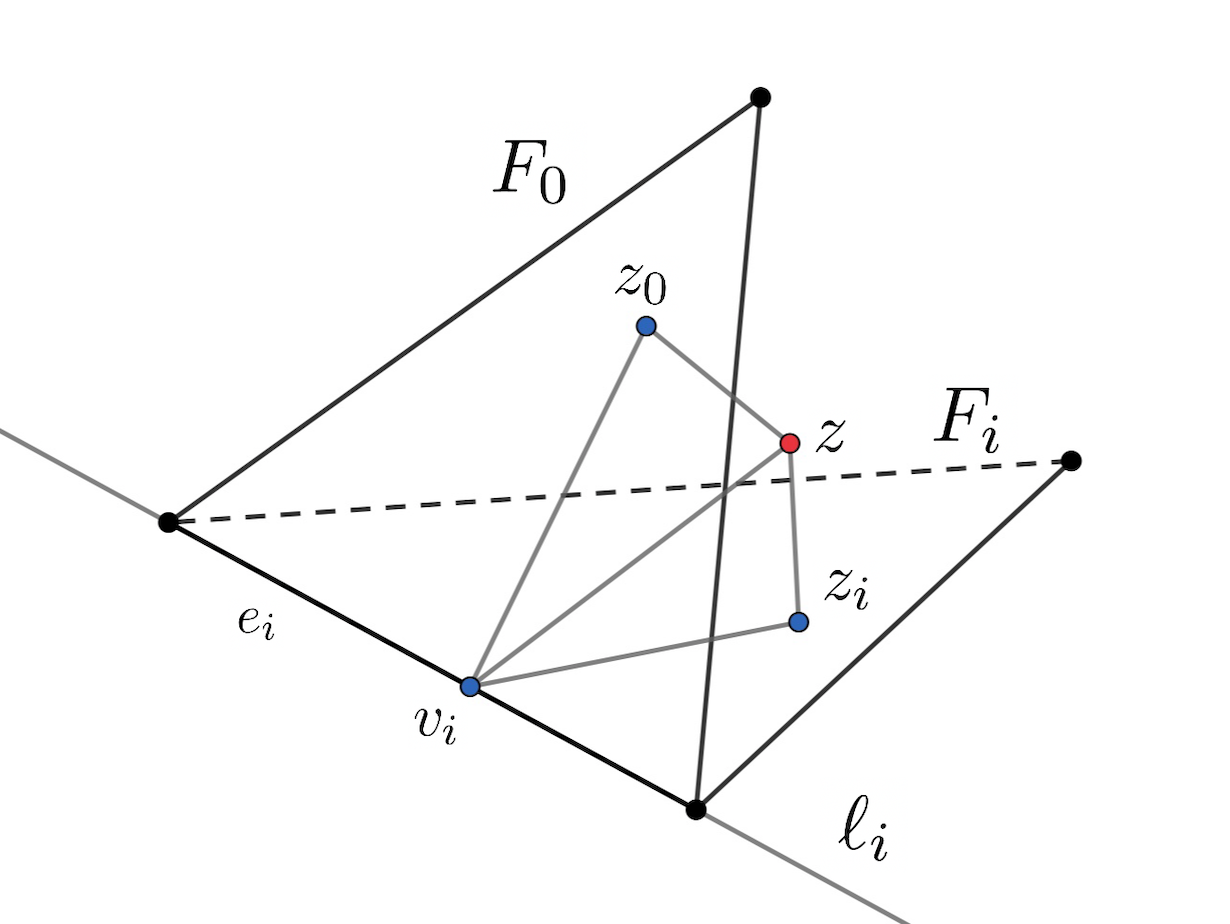}
          \caption{The projection of $z_1$ on face $F_0$ and $F_1$.}
          \label{dihedral1}
      \end{subfigure}
      \hfill
      \begin{subfigure}[b]{0.45\textwidth}
          \centering
          \includegraphics[width=0.9\textwidth]{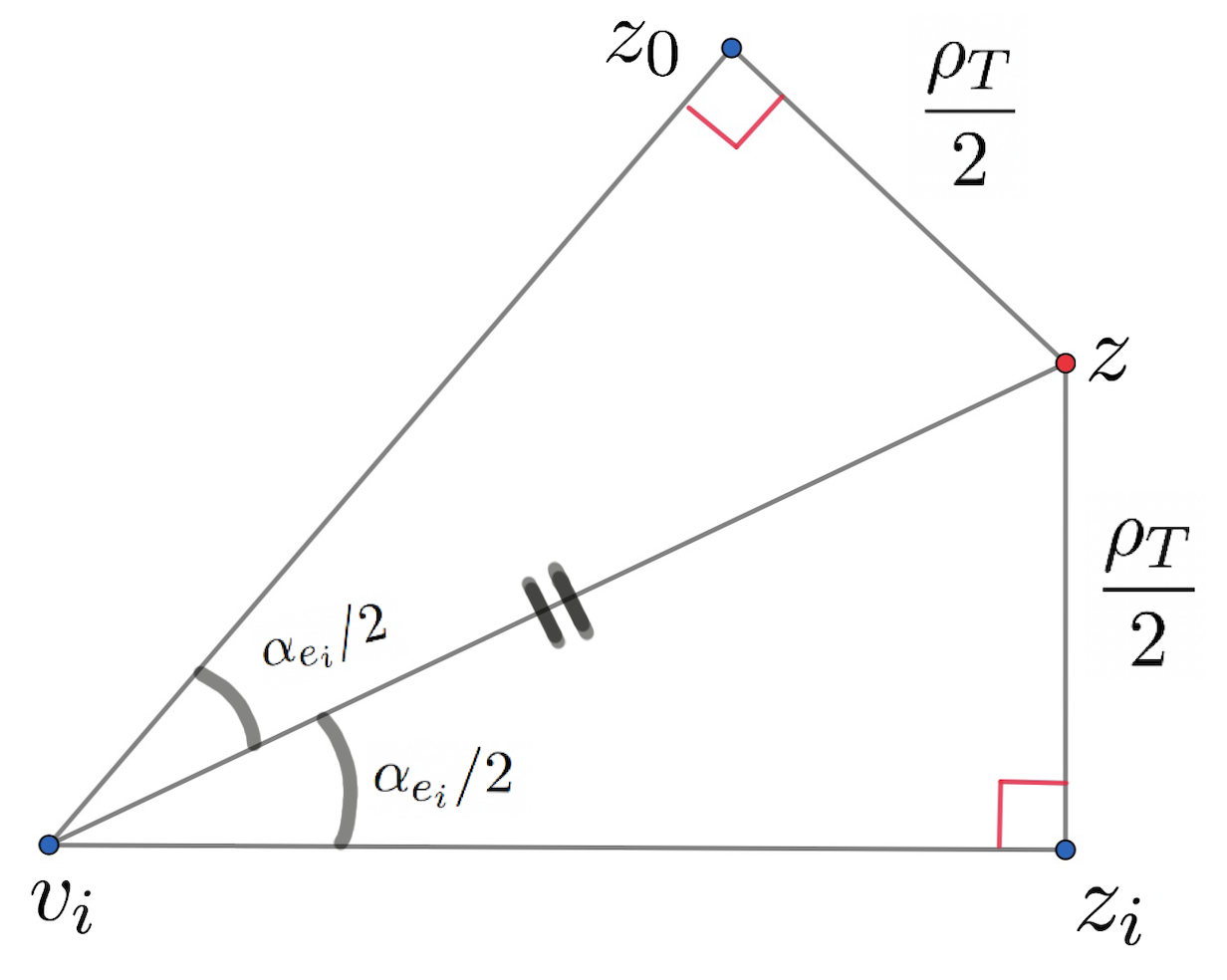}
          \caption{The plane through the incenter.}
          \label{dihedral2}
      \end{subfigure}
      \caption{A representation of the dihedral angle.}
    \label{figdihedral}
 \end{figure}
 
\MJN{We will also need the following result that bounds $\textnormal{dist}(z_{T,F},  \p F)$  from below using the dihedral angles.} 
\begin{lemma}\label{lem:zT2angle}
\MJN{Let $T$ be a tetrahedron, 
and for each face $F\in \Delta_2(T)$,
let $z_{T,F}$ denote the orthogonal projection 
of the incenter of $T$ onto $F$.  
Let 
$\alpha_e$ be the dihedral angle of $T$ with respect
to $e\in \Delta_1(T)$.}
We have 
\begin{equation}\label{equ:distFace}
        \min\limits_{F \in \Delta_2(T)}\textnormal{dist}(z_{T,F},  \p F) 
        \textcolor{black}{\ge } \min\limits_{e \in \Delta_1(T)} \frac{\rho_T}2 \sqrt{\frac{1+\cos(\alpha_e)}{1-\cos(\alpha_e)}}.
\end{equation}
\end{lemma}

\begin{proof}
\MJN{We use the short hand notation  depicted in 
Figure \ref{figdihedral}. In particular,  $z$ denotes the incenter of $T$ and  $F_i \in \Delta_2(T)$, $i=0,\ldots , 3$ denote the faces of $T$. Let $z_{i}$ be the orthogonal
projection of $z$ onto the plane containing $F_i$ and note that $|[z,z_i]|=\rho_T/2$.  We need to find a lower bound for ${\rm dist}(z_{k},\p F_{k})$ ($k=0, \dots, 3$) and without loss of generality we consider the case $k=0$.  To this end,  let $e_i= \partial {F_0} \cap \partial {F_i}$, $i=1,2,3$ and furthermore let $\ell_i$ be the line containing $e_i$.  Let $\gamma_i$ be the plane determined by the points $z$, $z_{0}$, and $z_{i}$ and let $v_i= \ell_i \cap \gamma_i$. Since $\ell_i \perp [z,z_i]$ for $j=0,i$, we have the line $\ell_i$ is perpendicular to the plane $\gamma_i$, and thus $\ell_i \perp [v_i,z_{j}]$ for $j=0,i$. This implies }
\[
\MJN{\textnormal{dist}(z_{j}, \ell_i) = |[z_{j},v_i]|,~ j=0,i,\ \text{ and }} \ 
    \MJN{\alpha_{e_i}:=\angle z_{0}v_iz_{i} = \angle z v_i z_{0}+\angle z v_i z_{i}.}
\]

\MJN{Next, note the properties 
$[z,z_{j}]\perp [z_{j},v_i]$ for $j=0,i$
and $|[z,z_{j}]| = \rho_{T}/2$ imply that 
the triangles $[z,v_i,z_{0}]$
and $[z,v_i,z_{i}]$ are congruent (see Figure \ref{dihedral2}).
Consequently, $\angle z v_i z_{0}=\angle z v_i z_{i}=\alpha_{e_i}/2$ and so}
\begin{equation}\label{equ:distEdge}
    \begin{split}
\MJN{\textnormal{dist}(z_{0}, \ell_i)  = |[z_{0},v_i]|
        = \frac{\rho_{T}/2}{\tan (\alpha_{e_i}/2)} = \frac{\rho_{T}}{2} \sqrt{\frac{1+\cos(\alpha_{e_i})}{1-\cos(\alpha_{e_i})}}.}
    \end{split}
\end{equation}
\MJN{The result now follows after using ${\rm dist}(z_{0},\p F_0) \ge \min\limits_{1\le i \le 3} {\rm dist}(z_{0},\ell_i)$.}
\end{proof}

\subsection{Worsey-Farin splits}\label{sec:WF splits}
\MJN{Let $\mathcal{T}_h$ be a three-dimensional  
triangulation without hanging nodes.}
We recall the construction of the Worsey-Farin refinement
of $\calT_h$ in the following definition \cite{worsey1987ann,lai2007spline,guzman2022exact}.

\begin{definition}\label{def:WF}
The Worsey-Farin refinement of $\calT_h$, denoted by $\mathcal{T}^{wf}_h$, is defined by the following two steps:
\begin{enumerate}
\item Connect the incenter $z_T$ of of each tetrahedron $T\in \calT_h$
 to its four vertices; 
\item For each interior face $F = \overline{T_1} \cap \overline{T_2}$ with $T_1$, $T_2 \in \calT_h$, let $m_F = L \cap F$ where $L = [{z_{T_1},z_{T_2}}]$, the line segment connecting 
 the incenter of $T_1$ and $T_2$; meanwhile, for a boundary face $F$ with $F = \overline{T}\cap \p\Omega$ with $T\in \calT_h$, let $m_F$ be the barycenter of $F$. We then connect $m_F$ to the three vertices of the face $F$ and to the incenters $z_{T_1}$ and $z_{T_2}$ (or $z_T$ for the boundary case). 
 \end{enumerate}
We see that this two-step procedure divides each $T\in \calT_h$ into $12$ subtetrahedra; we
denote the set of  these subtetrahedra by $T^{wf}$. 
\end{definition}
The result \cite[Lemma 16.24]{lai2007spline} ensures that the \MJN{three-dimensional} Worsey-Farin refinement is well-defined; in particular, the line segment
connecting the incenters of neighboring tetrahedra intersects their common 
face.

\begin{definition}\label{def:shapeR}
We define the shape regularity constant of the triangulation $\calT_h$ 
as
\[
c_0 = \max_{T\in \calT_h} \frac{h_T}{\rho_T}.
\]
\end{definition}

It is well-known that shape regularity of a mesh leads to bounded dihedral angles. 
To be self-contained, we present a proof here.
\begin{figure} 
        \centering
        \includegraphics[width=0.5\textwidth]{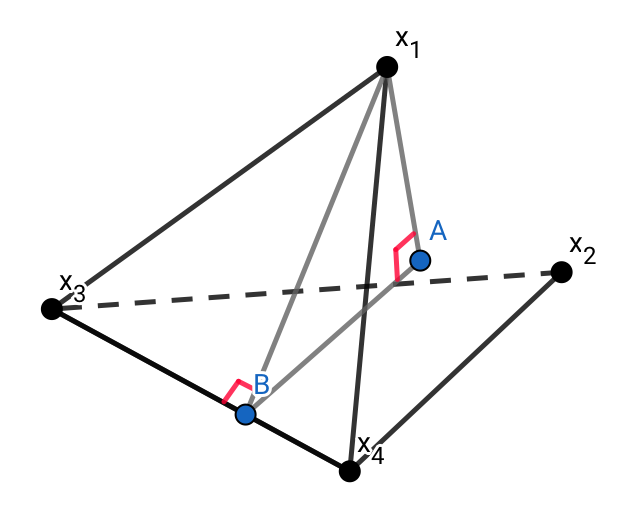}
        \caption{Computing dihedral angles.}
        \label{fig:dihedral}
\end{figure}

\begin{lemma} \label{lem:cos}
\MJN{Fix $T\in \calT_h$, and let $\alpha_e$ denote
the dihedral angle of $T$ with respect to $e\in \Delta_1(T)$.}
We then have 
\begin{equation} \label{equ:cos}
    |\cos (\alpha_e)| \le \sqrt{1-\textcolor{black}{c_0^{-2}}}\qquad \forall e\in \Delta_1(T).
\end{equation}
\end{lemma}
\begin{proof}
Write $T = [x_1,x_2,x_3,x_4]$,
consider the edge $e=[x_3,x_4]$, and let $\ell$ be the line containing $e$; see Figure \ref{fig:dihedral}.
Let $A$ be the orthogonal projection of \MJN{$x_1$}
onto the plane $\gamma$ containing the face $[x_2,x_3,x_4]$,
and let $B$ be the point on $\ell$
such that $[x_1,B]\perp \ell$. 
\MJN{Note that $[x_1,B]\perp\ell$ and $[x_1,A]\perp \ell$ implies $[A,B]\perp \ell$.}
Since $|[x_1,A]| \ge \rho_T$ by Proposition \ref{prop:height} and $|[x_1,B]| \le h_T$, the dihedral angle of $e$
satisfies
\[
\sin (\alpha_{e}) = \frac{|[x_1,A]|}{|[x_1,B]|} \ge \textcolor{black}{\frac{\rho_T}{h_T} \ge c_0^{-1}}.
\]
Therefore, we have $|\cos (\alpha_e)|=\sqrt{1-\sin^2 (\alpha_e)} \le \sqrt{1-\textcolor{black}{c_0^{-2}}}$.
\end{proof}

\section{Analysis of the shape regularity of Worsey-Farin splits}\label{sec:analysis}

In this section, we prove the main result of this note. We prove that the Worsey-Farin refinement
$\calT_h^{wf}$ is shape regular provided the parent triangulationn 
$\calT_h$ is shape regular. To be more precise, the following theorem will be proved:
\begin{thm} \label{thm:WFshape}
There exists a constant $c_1 >0$ only depending on $c_0$, the 
shape regularity constant of $\calT_h$ given in Definition \ref{def:shapeR} such that 
\[
\max\limits_{K \in \calT^{wf}_h} \frac{h_K}{\rho_K} \le c_1.
\]
\end{thm}
\MJN{For an explicit formula of $c_1$,} see \eqref{eqn:c1Def} and \eqref{equ:c2}.

\subsection{Local geometry}

\begin{figure}
      \centering
      \begin{subfigure}[b]{0.45\textwidth}
          \centering
          \includegraphics[width=\textwidth]{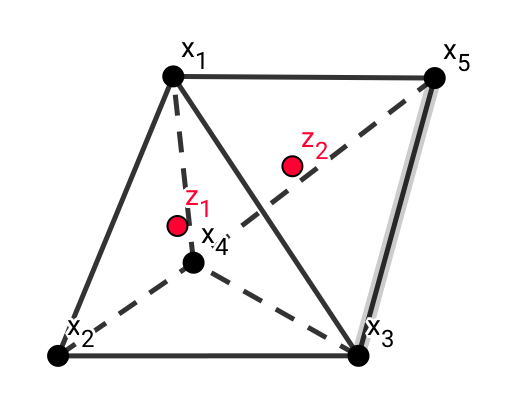}
          \caption{Two adjacent elements of the mesh.}
          \label{macro}
      \end{subfigure}
      \hfill
      \begin{subfigure}[b]{0.45\textwidth}
          \centering
          \includegraphics[width=\textwidth]{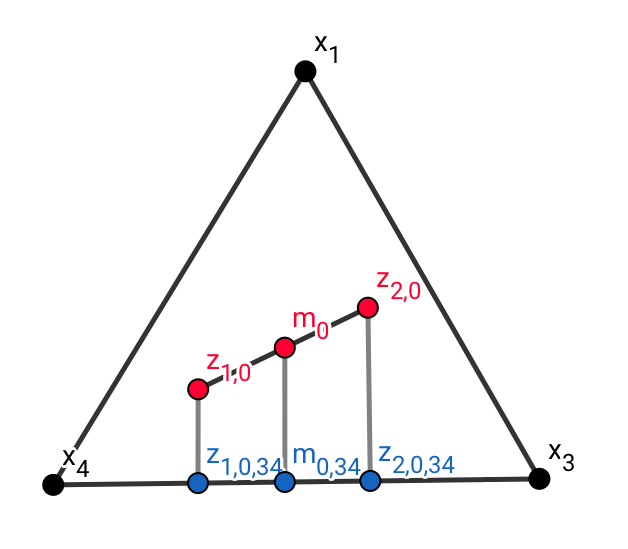}
          \caption{The common face}
          \label{face}
      \end{subfigure}
      \label{figmain}
      \caption{A representation of the Worsey-Farin splits.}
 \end{figure}
To prove the above theorem, we need to consider two cases: interior and boundary faces of $\calT_h$. 
The case of boundary faces is simpler, so we first focus on the interior faces. For that case, it is sufficient to consider two adjacent elements of the mesh $\calT_h$. 
To this end, let $T_1,T_2\in \calT_h$ be \MJN{two tetrahedra}
that share a common face $F_0$.
We write 
$T_1 = [x_1,x_2,x_3,x_4]$, $T_2 = 
[x_1,x_3,x_4,x_5]$, so that the common face is $F_0 = [x_1,x_3,x_4]$.
We further set $F_1= [x_2,x_3,x_4]$, 
and let $z_i$ be the incenter of $T_i$, $i=1,2$ (see Figure \ref{macro}).
For $i=1,2$, we denote by $z_{i,0}$ the orthogonal projections 
of $z_i$ onto the plane containing the face $F_0$ (see Figure \ref{face}).
Likewise the orthogonal projection of $z_1$ onto the 
plane containing the face $F_1$ 
is denoted by $z_{1,1}$ (see Figure \ref{dihedral1}). 
We denote the split point of
the face $F_0$ by $m_0$, i.e.,
$m_0$ is the intersection of 
the line $[z_1,z_2]$ and $F_0$.

\subsection{The position of split points and bounded dihedral angles}

The following proposition shows the relation between the split point $m_0$ and the projections $z_{i,0}$, $i=1,2$ of \MJN{the} incenter on the face $F_0$. 
\begin{prop} \label{prop:mF2zT}
The orthogonal projections $z_{i,0}\ (i=1,2)$ lie in the interior of $F_0$,
and the split point $m_0$ lies on the line segment $[z_{1,0}, z_{2,0}]$. Furthermore, we have
\[
{\rm dist}(m_0,\p F_0) \ge \min\limits_{i=1,2} {\rm dist}(z_{i,0}, \p F_0).
\]
\end{prop}
\begin{proof}
The proof of \cite[Lemma 16.24]{lai2007spline} shows that  $m_0$ lies on the line segment $[z_{1,0}, z_{2,0}]$ and that $z_{i,0}$ ($i=1,2$) lie in the interior of \MJN{$F_0$}. 

Let $\ell_i$, $i=1,2,3$ denote the lines that contain the 
three edges of $F_0$.
Because $m_0$ lies on the interior of the line segment  $[z_{1,0}, z_{2,0}]$,
there exists a constant $\theta \in (0,1)$ such that $m_0 = \theta z_{1,0} +(1-\theta) z_{2,0}$. 
Then by constructing similar triangles, we have
\begin{align*}
{\rm dist}(m_0,\p F_0) 
&= \min\limits_{1\le i \le 3} \textnormal{dist}(m_0, \ell_i)\\
&= \min\limits_{1\le i \le 3} \Big(
\theta  \textnormal{dist}(z_{1,0}, \ell_i)+(1-\theta)\textnormal{dist}(z_{2,0}, \ell_i)\Big)\\
&\ge \theta \min\limits_{1\le i \le 3}  \textnormal{dist}(z_{1,0}, \ell_i)+(1-\theta)\min_{1\le i\le 3}\textnormal{dist}(z_{2,0}, \ell_i)\\
&= \theta  \textnormal{dist}(z_{1,0}, \p F_0)+(1-\theta)\textnormal{dist}(z_{2,0}, \p F_0)\\
&\ge \min_{i=1,2}{\rm dist}(z_{i,0},\p F_0).
\end{align*}
\end{proof}



Combining \MJN{Lemma \ref{lem:zT2angle}, Lemma \ref{lem:cos}  and Proposition \ref{prop:mF2zT}}, we have the following lemma which describes the position of split points. We also include the case for boundary faces.
\begin{lemma} \label{lem:splitF}
\MJN{Recall that $m_F$ is the split point of $F$ constructed
by the Worsey-Farin split defined in Definition \ref{def:WF}.}
For any face $F$ of  $\calT_h$,  
\begin{equation*}
  {\rm dist}(m_F,\p F)\ge c_2 \mathop{\min_{T \in \calT_h}}_{F \in \Delta_2(T)} h_T,   
\end{equation*}
where
\begin{equation}\label{equ:c2}
c_2 := \min\{\mathfrak{c}_2, (3c_0)^{-1}\}, \quad  \mathfrak{c}_2:=(2c_0)^{-1}\sqrt{-1+\frac{2}{1+\sqrt{1-c_0^{-2}}}}.
\end{equation}

\end{lemma}
\begin{proof}
(i) \textbf{$F$ is an interior face.}  In this case $F \in \Delta_2(T)$ and $F \in \Delta_2(T')$ 
for some $T,T'\in \calT_h$. Without loss of generality, we assume $\textnormal{dist}(z_{T',F},  \p F) \ge \textnormal{dist}(z_{T,F},  \p F)$.  
\MJN{Lemma \ref{lem:zT2angle} and Proposition \ref{prop:mF2zT}} tell us that
\begin{align*}
\textnormal{dist}(m_F,  \p F) 
& \ge \textnormal{dist}(z_{T,F},  \p F) \\
& \ge 
\min\limits_{e \in \Delta_1(T)} \frac{\rho_T}{2} \sqrt{\frac{1+\cos(\alpha_e)}{1-\cos(\alpha_e)}}
= \min_{e\in \Delta_1(T)} \frac{\rho_T}2 \sqrt{-1+\frac{2}{1-\cos(\alpha_e)}}.
\end{align*}
If $\cos(\alpha_e)\ge 0$, then $\frac2{1-\cos(\alpha_e)}\ge 2$,
and if $\cos(\alpha_e)\le 0$, then $\frac2{1-\cos(\alpha_e)} = \frac2{1+|\cos(\alpha_e)|}\ge \frac2{1+\sqrt{1-c_0^{-2}}}$ 
by Lemma \ref{lem:cos}.
Consequently,
\[
\textnormal{dist}(m_F,  \p F) \ge 
\min_{e\in \Delta_1(T)}\frac{\rho_T}{2} \sqrt{-1+\frac{2}{1-\cos(\alpha_e)}} \ge \mathfrak{c}_2h_T.
\]

(ii) \textbf{$F$ is a boundary face.} 
Let \MJN{$T=[x_1, x_2, x_3, x_4]$} and  $F=[x_1,x_3,x_4]$, \MJN{and consider an arbitrary $e \in \Delta_1(F)$ with $\ell$ denoting the line containing $e$}. \MJN{Without loss of generality we assume $e= [x_3,x_4]$ and} adopt the notation in the proof of Lemma \ref{lem:cos}; see Figure \ref{fig:dihedral}.  
Because $m_F$ is the barycenter of $F$, we have 
\[ \frac{1}{3}|F|=\frac{1}{2} \textnormal{dist}(m_F, \ell)|e|.
\]
Moreover, clearly
\[
|F|=\frac{1}{2} |e||[x_1,B].
\]
And therefore, since $|[x_1,B]| \ge |[x_1,A]| > \rho_T$, (where we used \eqref{313} and the right triangle $[x_1,A,B]$) we get
\[ \textnormal{dist}(m_F, \ell) = \frac{1}{3}|[x_1,B]| \ge \frac{1}{3}\rho_T \ge  (3c_0)^{-1} h_T. 
\]
\MJN{Since $e \in \Delta_1(F)$ was arbitrary the result follows.}
\end{proof}

\subsection{Proof of Theorem \ref{thm:WFshape}}
Now we are ready to use Lemma \ref{lem:splitF} to prove Theorem \ref{thm:WFshape}.
\begin{proof}[Proof of Theorem \ref{thm:WFshape}.]
Let $K\in \calT^{wf}_h$, and
let $T\in \calT_h$ such that $K\in T^{wf}$.
We write $T = [x_1,x_2,x_3,x_4]$,
and assume, without loss of generality,
that $e:=[x_1,x_2]$ is an edge
of both $T$ and $K$.
Let $F\in \Delta_2(T)$ such that
the split point $m_F$ is a vertex of $K$.
In particular, $e \in \Delta_1(F)$ and $K=[x_1,x_2,m_F,z_T]$, where $z_T$
is the incenter of $T$.  We futher denote
by $\ell$, the line containing the edge $e$.

We again adopt the notation in the proof
of Lemma \ref{lem:cos} and refer to Figure \ref{fig:dihedral}.
Note that $[x_1,A]$ is normal to the plane $\gamma$ containing $[x_2,x_3,x_4]$, 
in particular, $[x_1,A] \perp [A,x_2]$.
Thus $|e|=|[x_1,x_2]|>|[x_1,A]|>\rho_T$ by \eqref{313}.
Now we have $h_K \le h_T$, $\rho_T< |e| \le h_T$ and,  by Lemma \ref{lem:splitF}, 
the volume $K$ is
\begin{equation}\label{equ:volK}
\begin{split}
    |K| &= \frac{1}{3}\frac{\rho_T}2 \times |[x_1,x_2,m_F]| 
    =\frac{1}{12} \rho_T |e| \textnormal{dist}(m_F, \ell)\\
    &\ge \frac{1}{12} \rho_T^2 \textnormal{dist}(m_F,  \p F)
    \ge \frac{c_2}{12} \rho_T^2 \big( \mathop{\min_{T' \in \calT_h}}_{F \in \Delta_2(T')}  h_{T'}\big) \ge  \frac{c_2}{12} \rho_T^3 \ge \frac{c_2}{12 c_0^3} h_T^3.
    \end{split}
\end{equation}
Here we also used 
\begin{equation*}
\mathop{\min_{T' \in \calT_h}}_{F \in \Delta_2(T')} h_{T'} \ge |e| > \rho_T.
\end{equation*}
Additionally, each face of $K$ is contained in a circle with radius $h_K/2$, and thus we have 
\begin{equation}\label{equ:areaF}
\sum_{F\in \Delta_2(K)} |F|\le \sum_{F\in \Delta_2(K)} \frac{\pi h_K^2}{4} = \pi h_K^2.
\end{equation}
Consequently, with Proposition \ref{prop:rhoK}, \eqref{equ:volK} and \eqref{equ:areaF}, we have
\begin{equation*}
\begin{split}
    \rho_K & = \frac{6|K|}{\sum\limits_{F \in \Delta_2(K)} |F|} \ge \frac{c_2 h_T^3 }{2 \pi \textcolor{black}{c_0^{3}} h_K^2} \ge \frac{c_2 h_K }{2 \pi \textcolor{black}{c_0^{3}}}.
\end{split}
\end{equation*}
Thus, setting
\begin{equation}\label{eqn:c1Def}
c_1 = \textcolor{black}{\frac{2\pi c_0^3}{c_2}}, 
\end{equation}
we have $\frac{h_K}{\rho_K}\le c_1$.  Because $K\in \calT^{wf}_h$ was arbitrary,
we conclude $\max\limits_{K\in \calT_h^{wf}} \frac{h_K}{\rho_K}\le c_1$.
\end{proof}

\section{Conclusion}
We have settled a conjecture concerning the shape regularity of a Worsey-Farin refinement of a parent triangulation. As described in the introduction, this is a crucial bound to obtain approximation results for splines; see for example   \cite[Theorem 18.15]{lai2007spline}. However, based on initial numerical calculations, the  constant $c_1$ in Theorem \ref{thm:WFshape} that relates the shape regularity of the parent
triangulation (i.e., $c_0$) and its Worsey-Farin refinement
is most likely not sharp.  In particular, the theorem suggest that $c_1$ scales like $c_0^5$ which could be quite 
large
even for a good quality parent triangulation. 
We hope that this work leads to further investigations and sharper estimates will emerge.      

 \bibliographystyle{siam}
 \bibliography{ref}
 

\end{document}